\documentclass{siamltex}
\usepackage{amsmath,amssymb,amsfonts,latexsym,graphicx,color,multirow,footnote}

\title{A contour-integral based QZ algorithm for generalized eigenvalue problems}

\author{Guojian Yin\thanks{Shenzhen Institutes of Advanced Technology, Chinese Academy of Sciences, Shenzhen, P.R. China ({\tt guojianyin@gmail.com}).} }

\begin{document}
\maketitle

\begin{abstract}
Recently, a kind of eigensolvers based on  contour integral were developed for computing the eigenvalues inside a given region in the complex plane. The CIRR method is a classic example among this kind of methods. In this paper, we propose a contour-integral based QZ method which is also devoted to computing partial spectrum of  generalized eigenvalue problems. Our new method takes advantage of the technique in the CIRR method of constructing a particular subspace containing the eigenspace of interest via contour integrals. The main difference between our method and CIRR is the mechanism of extracting the desired eigenpairs. We establish the related framework and address some implementation issues so as to make the resulting method applicable in practical implementations. Numerical experiments are reported to illustrate the numerical performance of our new method.
\end{abstract}



\section{Introduction}
Let $A$ and $B$ be large $n\times n$ matrices. Assume that we have a generalized eigenvalue problem 
\begin{equation}\label{eq:1-1}
A{\bf x} = \lambda B{\bf x},
\end{equation}
and want to compute the eigenvalues $\lambda_i$, along with their eigenvectors ${\bf x}_i$, of (\ref{eq:1-1}) inside a given region in the complex plane.
This problem arises in various areas of scientific and engineering applications,
for example in the model reduction of a linear dynamical system, one needs to know the response over a range of frequencies, see \cite{BDDRV00, GGD96, Ruhe98}. Computing a number of interior eigenvalues of a large problem remains one of the most difficult problems in computational linear algebra today \cite{FS12}. In practice, the methods of choice are always based on the projection techniques, the key to the success of which is to construct an approximately invariant subspace enclosing the eigenspace of interest.  The Krylov subspace methods in conjunction with 
 the spectral transformation techniques, such as the shift-and-invert technique, are most often used \cite{saad, stewart}.

Recently, the eigensolvers based on contour integral were developed to compute the eigenvalues inside a prescribed domain in the complex plane. The best-known methods of this kind are the Sakurai-Sugiura (SS) method \cite{ss} and the FEAST algorithm \cite{polizzi}. A major computational advantage of these contour-integral based methods is that they can be easily implemented in modern distributed parallel computers \cite{AT14, ISN10}. The FEAST algorithm works under the conditions that matrices $A$ and $B$ are Hermitian and $B$ is positive definite. In the SS method, the original eigenproblem (\ref{eq:1-1}) is reduced to a small one with Hankel matrices, if the number of sought-after eigenvalues is small. However, since Hankel matrices are usually ill-conditioned \cite{BGL07}, the SS method always suffers from numerical instability \cite{AT14, ST07}. By noticing this fact,  later in \cite{ST07}, Sakurai {\it et al.}  used the Rayleigh-Ritz procedure to replace the Hankel matrix approach to get a more stable algorithm called CIRR.

Originally, the CIRR method was formulated under the assumptions that matrices $A$ and $B$ are Hermitian and $B$ is positive definite, i.e., (\ref{eq:1-1}) is a Hermitian problem \cite{BDDRV00}. Moreover, it is required that the eigenvalues of interest are distinct. In \cite{IS10}, the authors adapted the CIRR method to non-Hermitian cases; meanwhile, they presented a block version of the CIRR method so as to deal with the degenerate systems.

The CIRR method is always accurate and powerful. It first constructs a subspace containing the eigenspace of interest through a sequence of particular contour integrals. Then the orthogonal projection technique is used to extract desired  eigenpairs.  In our work, we propose a contour-integral based QZ method for solving partial spectrum of (\ref{eq:1-1}). The motivation stems from the attempt of using the oblique projection method, instead of the orthogonal one, to extract desired eigenpairs in the CIRR method. When using oblique projection technique, the most important task is to find an appropriate left subspace, we borrow ideas of the JDQZ method \cite{FSV98}, and derive our new method. We establish the related mathematical framework. Some implementation issues will also be discussed before giving the resulting algorithm.

The rest of the paper is organized as follows. In Section 2,
we briefly review  the CIRR method \cite{ST07}. In Section 3, we derive a contour-integral based QZ method and establish the related mathematical framework. Then we will discuss some implementation issues and present the complete algorithm. Numerical experiments are reported in Section 4 to illustrate the numerical performance of our new method.

Throughout the paper, we use the following notation and terminology.
The subspace spanned by the columns of matrix $ X$ is
denoted by ${\rm span}\{ X\}$. The rank of  matrix $A$ is denoted by $\rank(A)$.
For any matrix $S$, we denote the submatrix
that lies in the first $i$ rows and the first $j$
columns of $S$ by $ S_{(1:i,1:j)}$, the submatrix consisting
of the first $j$ columns of $S$ by $ S_{(:,1:j)}$, and the submatrix
consisting of the first $i$ rows of $S$ by $ S_{(1:i,:)}$. The algorithms are presented in \textsc{Matlab} style. 

\section{The CIRR method}
In \cite{ss}, Sakurai {\it et al.} used a moment-based technique to formulate a contour-integral based method, i.e., the SS method, for finding the eigenvalues of (\ref{eq:1-1}) inside a given region. In order to improve the numerical stability of the SS method, a variant of it used the Rayleigh-Ritz procedure to extract desired  eigenpairs. This leads to the so-called CIRR method \cite{IS10, ST07}. Originally the CIRR method was derived in \cite{ST07} under the assumptions that (i) matrices $A$ and $B$ are Hermitian with $B$ being positive definite, and (ii) the eigenvalues inside the given region are distinct. In \cite{IS10}, the authors adapted the CIRR method to the non-Hermitian cases, meanwhile, a block version was proposed to deal with the degenerate problems. In this section we give a briefly review of the block CIRR method.

The matrix pencil $ z B-A$ is regular if ${\rm det}(zB-A)$ is not identically zero for all $z \in \mathbb{C}$ \cite{Lapack, demmel}. The  Weierstrass canonical form of  regular matrix pencil $zB-A$ is defined as follows. 

\begin{theorem} [\cite{G59}]
Let $zB-A$ be a regular matrix pencil of order $n$. Then there exist nonsingular matrices
$S$ and $T \in \mathbb{C}^{n\times n}$ such that
\begin{equation}\label{eq:8-3-1}
TAS = \begin{bmatrix}
  J_d    & 0   \\
   0   & I_{n-d}
\end{bmatrix}  \quad {\rm and} \quad TBS= \begin{bmatrix}
   I_d   & 0   \\
  0    & N_{n-d}
\end{bmatrix},
\end{equation}
where $J_d$ is a $d\times d$ matrix in Jordan canonical form
with its diagonal entries corresponding to the eigenvalues of $zB-A$, $N_{n-d}$ is an $(n-d)\times (n-d)$ nilpotent matrix also in Jordan canonical form, and $I_d$ denotes the identity matrix of order $d$.
\end{theorem}

Let $J_d$ in (\ref{eq:8-3-1}) be of the form
\begin{equation}\label{equ:7-17-1}
J_d =   \left[ \begin{array}{cccc}
 J_{d_1}(\lambda_1) & 0 & \cdots & 0\\
0 &  J_{d_2}(\lambda_2) & \cdots & 0\\
\vdots & \vdots& \ddots &  \vdots\\
0 & 0 & \cdots &  J_{d_m}(\lambda_m)
\end{array}
\right],
\end{equation}
where $\sum_{i=1}^m d_i  =d$, $1 \le d_i\le d$ for $i=1,\ldots m$ and
$ J_{d_i}(\lambda_i)$ are $d_i \times d_i$ matrices of the form
$$\begin{array}{ccc}
 J_{d_i}(\lambda_i) = \left[ \begin{array}{ccccc}
\lambda_i & 1 & 0& \cdots  & 0\\
0 & \lambda_i & 1 & & \vdots\\
 & \ddots & \ddots & \ddots &0\\
\vdots & & \ddots &\ddots &  1\\
0 &  \cdots && 0&\lambda_i
\end{array}
\right], & & i = 1, 2, \ldots, m
\end{array}
$$
with $\lambda_i$ being the eigenvalues. Here the $\lambda_i$
are not necessarily distinct and can be repeated according to
their multiplicities.

Let us partition $S$ into block form 
\begin{equation}\label{eq:3-4-1}
S = [S_1, S_2, \ldots, S_ m, S_{m+1}],
\end{equation}
where $S_i\in \mathbb{C}^{n\times d_i}$, $1\le i\le  m$, and $S_{m+1}\in \mathbb{C}^{n\times (n-d)}$. Then the first column in each $S_i$ is an eigenvector associated with eigenvalue $\lambda_i$ for $i = 1,\ldots, m$ \cite{BDDRV00, IS10, ISN10, YCY14}.

Let $\Gamma$ be a given positively
oriented simple closed curve in the complex plane. Below we show how to use the block CIRR method to compute the eigenvalues of (\ref{eq:1-1}) inside $\Gamma$, along with their associated eigenvectors.  Without loss of generality, let the set of eigenvalues of (\ref{eq:1-1}) enclosed by $\Gamma$ be
$\{\lambda_1, \ldots, \lambda_l\}$, and $s: = d_1+d_2+\cdots + d_{l}$ be the number of eigenvalues inside $\Gamma$ with multiplicity taken into account. 

Define the contour integrals 
\begin{equation} \label{eq:2-2-21}
 F_k:= \dfrac{1}{2\pi \sqrt{-1}}\oint_{\Gamma}z^k (z B- A)^{-1} Bdz,\quad k=0,1, \ldots. 
\end{equation}
With the help of residue theorem in complex analysis \cite{Rudin}, it was shown in \cite{ISN10} that 
\begin{equation}\label{eq:2-19-1}
F_k = S_{(:,1:s)}(J_{(1:s, 1:s)})^k(S^{-1})_{(1:s,:)}, \quad k = 0, 1, \ldots.
\end{equation}
Let $h$ and $g$ be two positive integers satisfying $hg\geqslant s$, and $Y$ be an $n\times h$ random matrix. Define 
\begin{equation}\label{eq:4-3-1}
U_k := F_k Y,\  k = 0,\ldots, g-1,\ {\rm and} \ U := [U_0, U_1,\ldots, U_{g-1}].
\end{equation}

We have the following result for the CIRR method.
\begin{theorem}\label{Th:4-7-1}
Let the eigenvalues inside $\Gamma$ be $\lambda_1, \ldots, \lambda_l$, then the number of eigenvalues of (\ref{eq:1-1}) inside $\Gamma$ is $s$, counting multiplicity. If $\rank(U)=s$, then we have
\begin{equation}\label{eq:4-4-2}
{\rm span}\{U\}= {\rm span} \{S_{(:, 1:s)}\}.
\end{equation} 
\end{theorem}
\begin{proof} 
By (\ref{eq:2-19-1}) and (\ref{eq:4-3-1}), we know that
\begin{equation}\label{eq:4-7-2}
U = S_{(:,1:s)} E, 
\end{equation}
where 
\begin{equation}\label{eq:4-9-1}
 E =\left[(S^{-1})_{(1:s,:)}Y, J_{(1:s, 1:s)}(S^{-1})_{(1:s,:)}Y,\ldots,  (J_{(1:s, 1:s)})^{g-1}(S^{-1})_{(1:s,:)}Y \right].
\end{equation}
Since the rank of $U$ is $s$, we have that $E$ is full-rank, following from which the expression (\ref{eq:4-4-2}) holds. \end{proof}

According to Theorem \ref{Th:4-7-1},  we know that ${\rm span}\{U\}$ contains the eigenspace corresponding to the desired eigenvalues. The block CIRR method uses the well-known orthogonal projection technique to extract the eigenpairs inside $\Gamma$ from ${\rm span}\{U\}$, i.e., imposing the Ritz-Galerkin condition:
\begin{equation}\label{eq:4-7-5}
A{\bf x}-\lambda B{\bf x}\  \bot\  {\rm span}\{U\},
\end{equation}
where $\lambda\in \mathbb{C}$ and ${\bf x}\in {\rm span}\{U\}$.

The main task of the block CIRR method is to evaluate $U_k$ (cf. (\ref{eq:4-3-1})). In practice, $U_k$ have to be computed approximately by a numerical integration scheme:
\begin{equation}\label{eq:3-29-1}
U_k  \approx \tilde{U}_k =\dfrac{1}{2\pi \sqrt{-1}}\sum^{q}_{j=1}\omega_jz_j^{k}(z_j  B- A)^{-1} BY, \quad k = 0, 1,\ldots, g-1,
\end{equation}
where $z_j$ are the integration points and $\omega_j$ are the corresponding weights.
From (\ref{eq:3-29-1}), it is easy to see that the dominant work of the block CIRR method is actually solving $q$ generalized shifted linear systems of the form
\begin{equation}\label{eq:3-29-2}
(z_j B-A)X_j = BY, \quad j = 1, 2,\ldots, q.
\end{equation}
Noticing that the integration points $z_j$ and the columns of right-hand sides are independent, the CIRR method can be easily implemented in modern distributed parallel computer.

The complete block CIRR method is summarized as follows.

\vskip3mm
\hrule
\hrule
\vskip1mm
\noindent Algorithm 1: The block CIRR method
\vskip1mm
\hrule
\hrule
\vskip1mm
\begin{tabbing}
x\=xxx\= xxx\=xxx\=xxx\=xxx\=xxx\kill
{\bf Input}: $h, g, q, Y\in \mathbb{C}^{n\times h}$. \\
{\bf Output}: Approximate eigenpairs $(\hat{\lambda}_i, \hat{{\bf x}}_i)$, $\hat{\lambda}_i$ inside $\Gamma$.\\
\>1.\> Compute $\tilde{U}_k, k = 0, 1,\ldots, g-1,$ approximately by (\ref{eq:3-29-1}).\\
\>2.\> Compute the singular value decomposition of $\tilde{U} = [\tilde{U}_0,\ldots, \tilde{U}_{g-1}]: \tilde{U} =\hat{U}\Sigma\hat{V}$.\\ 
\>3.\> Set $\hat{A}=  \hat{U}^* A  \hat{U}$ and $\hat{B}= \hat{U}^* B  \hat{U}$.\\
\>4.\> Solve the generalized eigenproblem of size $hg$: $\hat{A} {{\bf y}}=\hat{\lambda}  \hat{B} {{\bf y}}$,
to obtain the \\
\>\>  eigenpairs $\{(\hat{\lambda}_i, {\bf y}_i)\}_{i = 1}^{hg}$.\\
\>5.\> Compute $\hat{\bf {x}}_i=  \hat{U}{{\bf y}}_i$, and select $s$ approximate eigenpairs $(\hat{\lambda}_i, \hat{\bf {x}}_i)$.
\end{tabbing}
\vskip1mm
\hrule
\hrule
 \vskip3mm


\section{A contour-integral based QZ algorithm} 
The contour-integral based methods are recent efforts for the eigenvalue problems. The CIRR method is a typical example among the methods of this kind. According to the brief description in the previous section, the basic idea of the block CIRR method can be summarized as follows: (i) constructing a particular subspace that contains the desired   eigenspace by means of a sequence of contour integrals (cf. (\ref{eq:2-2-21})), and (ii) using the orthogonal projection technique, with respect to the subspace ${\rm span}\{U\}$ (cf. (\ref{eq:4-3-1})), to extract the desired eigenpairs.  In this section, we will derive another contour-integral based eigensolver. The idea stems from the attempt to use the oblique projection technique to extract desired eigenvalues in the block CIRR method. Applying the oblique projection method, the key step is finding a suitable left subspace.  We find an appropriate left subspace via using the QZ method to generate a generalized Schur decomposition associated with the desired eigenvalues. This intention finally leads us to a contour-integral based QZ method for solving (\ref{eq:1-1}).  We call the resulting algorithm CIQZ for ease of reference.

In this section, we first detail the derivation of our contour-integral based QZ method. Later on, we discuss some implementation issues that our contour-integral based QZ method may encounter in the practical application, after that, we give the complete CIQZ method.

\subsection{The derivation of the CIQZ algorithm}
The CIRR method uses the orthogonal projection technique to extract the sought-after eigenpairs from ${\rm span}\{U\}$. Here we consider using the oblique projection technique \cite{BDDRV00, saad}, another class of projection method, to compute the desired eigenpairs. 

Since ${\rm span}\{U\}$ contains the eigenspace of interest,  it is natural to choose ${\rm span}\{U\}$ as the right subspace (or \emph{search subspace}).  The oblique projection technique extracts the desired eigenpairs from ${\rm span}\{U\}$ by imposing the Petrov-Galerkin condition, which requires orthogonality with respect to some left subspace (or \emph{test subspace}), say, ${\rm span}\{W\}$:
\begin{equation}\label{eq:3-22-1}
A{\bf x}-\lambda B{\bf x}\  \bot\  {\rm span}\{W\},
\end{equation}
where $\lambda$ is located inside $\Gamma$, ${\bf x} \in {\rm span}\{U\}$, and $W$ is an $n\times s$ orthogonal matrix. Let $V$ be an $n\times s$ matrix whose columns form an orthogonal basis of ${\rm span}\{U\}$.  The orthogonality condition (\ref{eq:3-22-1}) leads to
the projected eigenproblem
\begin{equation}\label{eq:3-22-2}
W^*AV{\bf y}=\lambda W^*BV{\bf y},
\end{equation}
where ${\bf y}\in \mathbb{C}^{s}$ satisfies ${\bf x} = V{\bf y}$.

Now our task is to seek an appropriate left subspace ${\rm span}\{W\}$. Our discussion begins with a partial generalized Schur form for matrix pair $(A, B)$.
\begin{definition}[\cite{FSV98}]
A partial generalized Schur form of dimension $s$ for a matrix pair $(A, B)$ is the decomposition
\begin{equation}\label{eq:3-20-1}
AQ_s = Z_sH_s, \quad BQ_s = Z_sG_s,
\end{equation}
where $Q_s$ and $Z_s$ are orthogonal $n\times s$ matrices, and $H_s$ and $G_s$ are upper triangular $s\times s$  matrices. A column $(Q_s)_{(:, i)}$ is referred to as a generalized Schur vector, and we refer to a pair $((Q_s)_{(:, i)}, (H_s)_{(i, i)}/(G_s)_{(i, i)})$ as a generalized Schur pair.
\end{definition}

The formulation (\ref{eq:3-20-1}) is equivalent to 
\begin{equation}\label{eq:3-20-2}
(Z_s)^*AQ_s = H_s, \quad (Z_s)^*BQ_s = G_s,
\end{equation}
from which we know that $(H_s)_{(i, i)}/(G_s)_{(i, i)}$ are the eigenvalues of $(H_s, G_s)$. Let ${\bf y}_i$ be the eigenvectors of pair $(H_s, G_s)$ associated with $(H_s)_{(i, i)}/(G_s)_{(i, i)}$, then we have $( (H_s)_{(i, i)}/(G_s)_{(i, i)}, Q_s {\bf y}_i)$ are the eigenpairs of $(A, B)$ \cite{FSV98, MS73}.

Applying the QZ algorithm to (\ref{eq:3-22-2}) to yield generalized Schur form 
\begin{equation}\label{eq:3-20-3}
(P_L)^*(W^*AV)P_R = H_A \quad and \quad (P_L)^*(W^*BV)P_R = H_B,
\end{equation}
where $P_R$ and $P_L$ are orthogonal $s\times s$ matrices,  $H_A$ and $H_B$ are upper triangular $s\times s$ matrices. The eigenvalues of pair $(W^*AV, W^*BV)$ are $\{(H_A)_{(i, i)}/(H_B)_{(i, i)}\}_{i}^s$ \cite{gvl, MS73}.

Comparing (\ref{eq:3-20-2}) with (\ref{eq:3-20-3}), it is readily to see that we have constructed a partial generalized Schur form in (\ref{eq:3-20-3}) for matrix pair $(A, B)$: $VP_R$ constructs a $Q_s$ and $WP_L$ constructs a $Z_s$. 

Since the desired eigenvalues are finite,  the diagonal entries of $H_A$ and $H_B$ are non-zero, which means that $H_A$ and $H_B$ are nonsingular.  In view of (\ref{eq:3-20-3}), we can conclude that 
\begin{equation}\label{eq:4-10-1}
{\rm span}\{WP_L\}={\rm span}\{AVP_R\}={\rm span}\{BVP_R\}.
\end{equation}
On the other hand, since $P_L$ and $P_R$ are nonsingular, we have 
\begin{equation}\label{eq:3-27-1}
{\rm span}\{W\}={\rm span}\{AV\}={\rm span}\{BV\}.
\end{equation}
Motivated by (\ref{eq:3-27-1}), we choose the left subspace ${\rm span}\{W \}$ to be ${\rm span}\{AU+BU \}$. Below we want to justify this choice.

\begin{theorem}\label{Th:2-13-1}
Let $L, D \in \mathbb{C}^{n\times t}, t\ge s$, be arbitrary matrices, and $R = F_0D$. A projected matrix pencil $z\hat{B}-\hat{A}$ is defined by $\hat{B}=L^*BR$ and $\hat{A}=L^*AR$. If ranks of both $L^*(T^{-1})_{(:,1:s)}$ and $(S^{-1})_{(1:s, :)}D$ are $s$, then the eigenvalues of $z\hat{B}-\hat{A}$ are $\lambda_1, \ldots, \lambda_l$, i.e., the eigenvalues that are located inside $\Gamma$.  
\end{theorem}

The proof is almost identical with that of Theorem 4 in \cite{IS10}, where the contour integrals $F_k$ were defined as $\frac{1}{2\pi \sqrt{-1}}\oint_{\Gamma}z^k (z B- A)^{-1} dz$, that is, the term $B$ was dropped comparing with the expression (\ref{eq:2-19-1}).

Theorem \ref{Th:2-13-1} says that the desired   eigenvalues $\{\lambda_i\}_{i=1}^l$ can be solved via computing the eigenvalues of projected eigenproblem $z\hat{B}-\hat{A}$, if the ranks of both $L^*(T^{-1})_{(:,1:s)}$ and $(S^{-1})_{(1:s, :)}D$ are $s$. Due to this, we want to show the following results.

\begin{theorem}\label{Th:4-9-1}
If the rank of $U$ is $s$, then the ranks of $( AU+ BU)^*(T^{-1})_{(:,1:s)}$ and $(S^{-1})_{(1:s, :)}U$ are $s$.
\end{theorem}
\begin{proof}
We first show that the rank of $(S^{-1})_{(1:s, :)}U$ is $s$. By (\ref{eq:8-3-1}) and (\ref{eq:4-7-2}), we have
\begin{equation}\label{eq:4-7-4}
(S^{-1})_{(1:s, :)}U = (S^{-1})_{(1:s, :)}S_{(:, 1:s)} E = E.
\end{equation}
Since $U$ is full-rank, by (\ref{eq:4-7-2}), we know that $\rank (E) = s$. Therefore, the rank of $(S^{-1})_{(1:s, :)}U$ is $s$.

Next we show that the rank of $( AU+ BU)^*(T^{-1})_{(:,1:s)}$ is $s$. For convenience, we turn to show that the rank of $((T^{-1})_{(:,1:s)})^*( AU+ BU)$, i.e., the conjugate transpose of $( AU+ BU)^*(T^{-1})_{(:,1:s)}$, is $s$. 

Since ${\rm span}\{AU\}={\rm span}\{BU\}$ (cf. (\ref{eq:3-27-1})), there exists a $hg \times hg$ nonsingular matrix $\Delta$ such that $AU = BU\Delta$. 
According to (\ref{eq:8-3-1}), (\ref{eq:2-19-1}),  and (\ref{eq:4-7-2}), we have
\begin{equation}\label{eq:3-26-1}
((T^{-1})_{(:,1:s)})^*( AU+ BU) = (BS_{(:, 1:s)})^*BS_{(:, 1:s)}E( \Delta+ I_s).
\end{equation}
In view of (\ref{eq:8-3-1}), we know $BS_{(:, 1:s)}$ is full-rank, which means $(BS_{(:, 1:s)})^*BS_{(:, 1:s)}$ is nonsingular. By (\ref{eq:3-26-1}), we can conclude that $(T^{-1})_{(:,1:s)})^*( AU+ BU)$ is full rank, thus the rank of $(AU+ BU)^*(T^{-1})_{(:,1:s)}$ is $s$.
\end{proof}

Based on Theorem \ref{Th:2-13-1} and Theorem \ref{Th:4-9-1}, we have that the eigenvalues of $(( AU+ BU)^*AU,( AU+ BU)^*BU)$ are  the eigenvalues of (\ref{eq:1-1}) inside $\Gamma$, which justifies our choice of taking the left subspace to be ${\rm span}\{AU+ BU\}$. On the other hand, the columns of $V$ and $W$ form the base of ${\rm span}\{AU+ BU\}$ and ${\rm span}\{U\}$, respectively.  As a consequence, there exist $hg\times hg$ nonsingular matrices $P_1$ and $P_2$, such that
\begin{equation}\label{eq:6-15-1}
( AU+ BU)P_1 = [W, 0],\quad UP_2 = [V, 0].
\end{equation}
 Now, we have 
\begin{equation}\label{eq:6-15-2}
P^*_1(z( AU+ BU)^*BU- ( AU+ BU)^*AU)P_2=\begin{bmatrix}
   zW^*BV-W^*AV   &  0  \\
    0  &  0
\end{bmatrix}.
\end{equation}
Therefore, $(W^*AV, W^*BV)$ shares the same eigenvalues with $(( AU+ BU)^*AU,( AU+ BU)^*BU)$, which are $\{(H_A)_{(i, i)}/(H_B)_{(i, i)}\}_{i=1}^s$ by (\ref{eq:3-20-3}).  Let  $((H_A)_{(i, i)}/(H_B)_{(i, i)}, \tilde{{\bf y}}_i)$ be the eigenpairs of $(H_A, H_B)$, then according to (\ref{eq:3-20-2}) and (\ref{eq:3-20-3}), we have that $((H_A)_{(i, i)}/(H_B)_{(i, i)}, VP_R\tilde{{\bf y}}_i)$ are exactly the eigenpairs of (\ref{eq:1-1}) inside $\Gamma$. 

We use the following algorithm to summarize the above discussion. 

\vskip3mm
\hrule
\hrule
\vskip1mm
\noindent Algorithm 2: A contour-integral based QZ algorithm.
\vskip1mm
\hrule
\hrule
\vskip1mm
\begin{tabbing}
x\=xxx\= xxx\=xxx\=xxx\=xxx\=xxx\kill
{\bf Input}: $h, g, q, Y\in \mathbb{C}^{n\times h}$. \\
{\bf Output}: Approximate eigenpairs $(\tilde{\lambda}_i, \tilde{{\bf x}}_i), i = 1, \ldots, s$.\\
\>1.\> Compute $\tilde{U}_k,  k =0, 1,\ldots, g-1, $ approximately by (\ref{eq:3-29-1}).\\
\>2.\> Form $\tilde{U} = [\tilde{U}_0, \tilde{U}_1, \ldots, U_{g-1}]$ and compute orthogonalization: \\
\>\> $V = \texttt{orth}(\tilde{U})$ and $W = \texttt{orth}(AV+BV)$.\\
\>3.\> Compute $\tilde{A}=  W^* A  V$ and $\tilde{B}=  W^* B  V$.\\
\>4.\> Compute $[ S_A, S_B, U_L, U_R, V_L, V_R] = \texttt{qz}(\tilde{A}, \tilde{B})$.\\
\>5.\> Compute $\tilde{\lambda}_i = (S_A)_{(i, i)}/(S_B)_{(i, i)}$ and $\tilde{{\bf x}}_i = V U_R(V_R)_{(:, i)}$.\\
\>6.\>  Select the approximate
eigenpairs $(\tilde{\lambda}_i, \tilde{{\bf x}}_i)$.
\end{tabbing}
\vskip1mm
\hrule
\hrule
 \vskip3mm

\subsection{The implementation issues}
If we apply Algorithm 2 to compute the eigenvalues inside $\Gamma$,  we will encounter some issues in practical implementation, just like other contour-integral based eigensolvers \cite{polizzi, ss, ST07}. In this section, we discuss the implementation issues of our new method.

The first issue we have to treat is about selecting a suitable size for the starting matrix $Y$, with a prescribed parameter $g$. Since $U$ (cf. \ref{eq:4-3-1}) is expected to span a subspace that contains the eigenspace of interest, we have to choose a parameter $h$, the number of columns of $Y$, such that $hg\geqslant s$, the number of eigenvalues inside $\Gamma$.  A strategy was proposed in \cite{SFT} for finding a suitable parameter $h$ for the block CIRR method. It starts with finding an estimation to $s$. Giving a positive integer $h_0$,  by ``$Y_{h_0}\sim \textsf{N}(0, 1)$", we mean $Y_{h_0}$ is an $n\times h_0$ matrix with i.i.d. entries drawn from standard normal distribution $\textsf{N}(0, 1)$. By (\ref{eq:2-19-1}) and (\ref{eq:4-3-1}), one can easily verify that the mean
\begin{equation}\label{eq:3-28-1}
\mathbb{E}[{\rm trace}((Y_{h_0})^*F_0Y_{h_0})]=h_0\cdot{\rm trace}(F_0)=h_0\cdot{\rm trace}(S_{(:, 1:s)}(S^{-1})_{(1:s, :)})=h_0\cdot s.
\end{equation}
Therefore, 
\begin{equation}\label{eq:4-12-1}
s_0 := \frac{1}{h_0}\cdot \mathbb{E}[{\rm trace}((Y_{h_0})^*F_0Y_{h_0})]
\end{equation}
gives an initial estimation to $s$ \cite{futa, YCY14}.  With this information on hand, the strategy in \cite{SFT} works as follows: (i) set $h=\lceil \frac{s_0 \kappa}{g}\rceil$, where $\kappa > 1$, (ii)  select the starting matrix $Y\in \mathbb{C}^{n\times h}$ and compute $\tilde{U}_k$ by (\ref{eq:3-29-1}), (iii) if the minimum singular value $\sigma_{\min}$ of $\tilde{U} = [\tilde{U}_0,\ldots, \tilde{U}_{g-1}]$ is small enough, we find a suitable $h$;  otherwise,  replace $h$ with $\kappa h$ and repeat (ii) and (iii). We observe that the formula (\ref{eq:4-12-1}) always gives a good estimation of $s$. However the computed $s_0$ may be much larger than $s$ in some cases, such as the matrices $A$ and $B$ are ill-conditioned, which leads to that it is potentially expensive to compute the singular value decomposition of $\tilde{U}$. Due to this fact, in our method we turn to use the strategy proposed in \cite{YCY14}, whose working mechanism is as follows: use the rank-revealing QR factorization \cite{tony, gvl} to monitor the numerical rank of $\tilde{U}$,
if $\tilde{U}$ is numerically rank-deficient, then it means that the subspace spanned by $\tilde{U}$ already contains the desired   eigenspace sufficiently, as a result, we find a suitable parameter $h$.

Another issue we have to address is designing the stopping criteria. The stopping criteria here include two aspects: (i) all computed approximate eigenpairs attain the prescribed accuracy, and (ii) all eigenpairs inside the given region are found. 

As for the first aspect of the stopping criteria, since we can only compute $U$ approximately by some quadrature scheme (cf. (\ref{eq:3-29-1})),  the approximate eigenpairs computed by Algorithm 2 may be unable to attain the prescribed accuracy in practical applications. A natural solution is to refine $\tilde{U}$  (step 2 in Algorithm 2)  iteratively. A refinement scheme was suggested in \cite{IDS15}. Let $\tilde{U}_0^{(0)}= Y$ and $l$ be a positive integer, the refinement scheme iteratively computes $U^{(l)}_k= F_k\tilde{U}^{(l-1)}_ 0$ by a $q$-point numerical integration scheme:
\begin{equation}\label{eq:4-13-1}
U^{(l)}_k  \approx \tilde{U}^{(l)}_k =\dfrac{1}{2\pi \sqrt{-1}}\sum^{q}_{j=1}\omega_jz_j^{k}(z_j  B- A)^{-1} B\tilde{U}^{(l-1)}_ 0, \quad k = 0, 1,\ldots, g-1,
\end{equation}
and then constructs 
\begin{equation}\label{eq:4-13-2}
\tilde{U}^{(l)}=\left[\tilde{U}^{(l)}_0, \tilde{U}^{(l)}_1,\ldots, \tilde{U}^{(l)}_{g-1}\right].
\end{equation}
The refined $\tilde{U}^{(l)}$ is used to form projected eigenproblem (\ref{eq:3-22-2}),  through which we compute the approximate eigenpairs. The accuracy of approximate eigenpairs will be improved as the iterations proceed, see \cite{SFT} for more details. 

If all $s$ approximate eigenpairs attain the prescribed accuracy after a certain iteration, we could stop the iteration process. However, in general we do not know the number of eigenvalues inside the target region in advance.  This fact leads to the second aspect of the stopping criteria: how to guarantee that all desired   eigenpairs are found when the iteration process stops.  We take advantage of the idea proposed in \cite{YCY14}. The rationale of the idea is that, as the iteration process proceeds, the accuracy of desired  eigenpairs will be improved while the spurious ones do not, as a result, there will exist a gap of accuracy between the desired  eigenpairs and the spurious ones \cite{YCY14}. Based on this observation, a test tolerance $\eta$, say $1.0\times 10^{-3}$, is introduced to discriminate between the desired eigenpairs and the spurious ones. Specifically, for approximate eigenpair $(\tilde{\lambda}_i, \tilde{{\bf x}}_i)$, define the corresponding residual norm as 
\begin{equation}\label{eq:4-14-1}
r_i = \dfrac{\|A\tilde{{\bf x}}_i-\tilde{\lambda}_i B\tilde{{\bf x}}_i\|}{\|A\tilde{{\bf x}}_i\|+\|B\tilde{{\bf x}}_i\|}.
\end{equation}
If  $r_i <\eta$, then we view $(\tilde{\lambda}_i, \tilde{{\bf x}}_i)$ as an approximation to a sought-after eigenpair and refer to it as a filtered eigenpair by $\eta$. If the numbers of  filtered eigenpairs are the same in two consecutive iterations, then we set them to be the number of eigenvalues inside $\Gamma$, see \cite{YCY14} for more details.
 
  From (\ref{eq:4-13-1}) we can see that, in each iteration, the dominate work is to compute $q$ generalized shifted linear systems of the form
\begin{equation}\label{eq:5-11-1}
(z_i B-A)X_i^{(l-1)} = B U_0^{(l-1)}, \quad i = 1, 2, \ldots, q.
\end{equation}

Integrating the above strategies with Algorithm 2, below we give the complete CIQZ algorithm for computing the eigenpairs inside the given region $\Gamma$. 

\vskip3mm
\hrule
\hrule
\vskip1mm
\noindent Algorithm 3: The complete CIQZ method 
\vskip1mm
\hrule
\hrule
\vskip1mm
\begin{tabbing}
x\=xxx\= xxx\=xxx\=xxx\=xxx\=xxx\kill
{\bf Input}: $A, B, h_0, g, q, \kappa, \eta, \epsilon, \mathrm{max\_iter}$. \\
{\bf Output}: Approximate eigenpairs $(\tilde{\lambda}_i, \tilde{{\bf x}}_i), i = 1, \ldots, s$.\\
\>1.\>Let $Y_{h_0}\sim \textsf{N}(0, 1)$, compute $\tilde{U}_k, k = 0, \ldots, g-1,$ by (\ref{eq:3-29-1}).\\
\>2.\> Compute $s_0 = \lceil \frac{1}{h_0}{\rm trace}((Y_{h_0})^*\tilde{U}_0)\rceil$, and set $h=\max\{ \lceil \frac{s_0 \kappa}{g}\rceil, h_0\}$.\\
\>3.\> If $h> h_0$\\
\>4.\>\> Pick $\check{Y}_{h-h_0} \sim \textsf{N}(0, 1)$ and compute $\check{U}_k$ by by (\ref{eq:3-29-1}). Augment $\check{U}_k$ \\
\>\>\> to $\tilde{U}_k$: $\tilde{U}_k = \left[\tilde{U}_k, \check{U}_k \right]\in \mathbb{C}^{n\times h}$  and construct $\tilde{U}=\left[\tilde{U}_0, \tilde{U}_1,\ldots, \tilde{U}_{g-1} \right]$. \\
\>5.\> Else\\
\>6.\>\> Set $h = h_0$ and construct $\tilde{U}=\left[\tilde{U}_0, \tilde{U}_1,\ldots, \tilde{U}_{g-1} \right]$.\\
\>7.\> End\\
\>8.\> Compute the rank-revealing QR factorization: $\tilde{U}= VR\Pi$. Set $s_1= \rank(R)$.\\
\>\>  If $s_1 < hg$, stop; otherwise, set $h_0 = h$, $h = \kappa h$ and go to step 3.\\
\>9.\> Set $e(0) = 0$ and $c(0) = n$.\\
\>10.\> For $k = 1, 2, \ldots, \mathrm{max\_iter}$\\
\>11.\>\> Compute the orthogonalization: $V = \texttt{orth}(\tilde{U})$ and $W =  \texttt{orth}( AV+ BV)$.\\
\>12.\>\> Compute $\tilde{A}=  W^*A  V$ and $\tilde{B}=  W^* B  V$. Set $s_1= \rank(\tilde{A})$.\\
\>13.\>\> Compute $[ S_A, S_B, U_L, U_R, V_L, V_R] = \texttt{qz}(\tilde{A}, \tilde{B})$.\\
\>14.\>\> Compute $\tilde{\lambda}_i = (S_A)_{(i, i)}/(S_B)_{(i, i)}$ and $\tilde{{\bf x}}_i = V U_R(V_R)_{(:, i)}, i = 1, \ldots, s_1$.\\
\>15.\>\> Set $r = [\ ], \Lambda^{(k)}= [\ ], X^{(k)}=[\  ]$, and $c(k) = 0$.\\
\>16.\> \> For $i = 1:s_1$\\
\>17.\> \>\> Compute $r_i = \|A\tilde{{\bf x}}_i -\tilde{\lambda}_i B \tilde{{\bf x}}_i\|/(\|A\tilde{{\bf x}}_i\| +\| B \tilde{{\bf x}}_i\|)$.\\
\>18.\> \>\> If $\tilde{\lambda}_i$ inside $\Gamma$ and $r_i < \eta$, then $c(k) = c(k)+1, r = [r, r_i]$,\\
\>\> \>\> $X^{(k)}=[X^{(k)}, \tilde{\bf x}_i]$ and $\Lambda^{(k)}=[\Lambda^{(k)}, \tilde{\lambda}_i]$.\\
\>19.\>\>End\\
\>20.\> \>Set $e(k) = \max(r)$.\\
\>21.\>\>If $c(k) = c(k-1)$ and $e(k)< \epsilon$, set $\tilde{\lambda}_i = (\Lambda^{(k)})_{i}, \tilde{\bf x}_i =  (X^{(k)})_{(:, i)}$. Stop.\\
\>22.\>\> Set $Y = \tilde{U}_0$, and compute $\tilde{U}_k$ by (\ref{eq:3-29-1}).  Construct $\tilde{U}=\left[\tilde{U}_0, \tilde{U}_1,\ldots, \tilde{U}_{g-1} \right]$.\\
\>23.\> End

\end{tabbing}
\vskip1mm
\hrule
\hrule
 \vskip3mm
 Here we give some remarks on Algorithm 3.\begin{itemize}
  \item [1.] Steps 1 to 8 are devoted to determining a suitable parameter $h$ for the starting matrix $Y$. Meanwhile, a matrix $\tilde{U}$ is also generated. 
  \item [2.] The for-loop, steps 16 to 19, is used to detect the spurious eigenvalues. Only the approximate eigenpairs whose residual norms are less than $\eta$ are retained.
  \item [3.] Step 21 refers to the stopping criteria, which contain two aspects: (i) the number of filtered eigenpairs by $\eta$ is the same with the one in the previous iteration, and (ii) the residual norms of all filtered eigenpairs are less than the prescribed tolerance $\epsilon$.
  \item [4.] By (\ref{eq:3-27-1}), theoretically, the left subspace can be chosen either ${\rm span}\{AU\}$ or ${\rm span}\{BU\}$. However, in practical implementation $U$ can only be computed  by a quadrature scheme to get an approximation $\tilde{U}$, in Algorithm 3 we choose the left subspace to be ${\rm span}\{A\tilde{U}+B\tilde{U}\}$ so as to include the information of both ${\rm span}\{A\tilde{U}\}$ and ${\rm span}\{B\tilde{U}\}$. \end{itemize}


\section{Numerical Experiments}\label{sec:experiments}
In this section, we use some numerical experiments to illustrate the performance of our CIQZ method (Algorithm 3). The test problems are from the Matrix Market collection \cite{DGL89}. They are the real-world problems from scientific and engineering applications. The descriptions of the related matrices are presented in \textsc{Table} \ref{Tab:5-1}, where  \texttt{nnz} denotes the number of non-zero entries and cond denotes the condition numbers which are computed by Matlab function \texttt{condest}.
All computations are carried out in \textsc{Matlab} version R2014b on a MacBook with an Intel Core i5 2.5 GHz processor and 8 GB RAM.

\begin{table}
\caption{Test problems from Matrix Market that are used in our experiments.}
\footnotesize{
\noindent
\begin{tabular}{c|ll|cllc}
No.&Problem & Size &Matrix& \texttt{nnz} & Property & \texttt{cond}
\\ \hline
\multirow{2}{*}{1}& \multirow{2}{*}{\texttt{BFW398}}  & \multirow{2}{*}{$398$}& $A$ & $3678$&  real unsymmetric &$7.58\times 10^{3}$\\ 
&   &&$B$& $2910$& real symmetric indefinite &$3.64\times 10^{1}$\\ \hline
\multirow{2}{*}{2}& \multirow{2}{*}{\texttt{BFW782}} &  \multirow{2}{*}{$782$} & $A$ & $7514$& real unsymmetric &$4.63\times 10^{3}$\\
&  & & $B$& $5982$&  real symmetric indefinite &$3.05\times 10^{1}$\\ \hline
\multirow{2}{*}{3}& \multirow{2}{*}{\texttt{DWG961}} &  \multirow{2}{*}{$961$} & $A$ & $3405$& complex symmetric indefinite &Inf\\
&  & & $B$& $10591$&  complex symmetric indefinite &$3.21\times 10^{7}$\\ \hline
\multirow{2}{*}{4}& \multirow{2}{*}{\texttt{MHD1280}} &  \multirow{2}{*}{$1280$} & $A$ & $47906$& complex unsymmetric &$9.97\times 10^{24}$\\
&  & & $B$& $22778$& complex Hermitian &$5.99\times 10^{12}$\\ \hline
\multirow{2}{*}{5} & \multirow{2}{*}{\texttt{MHD3200}} & \multirow{2}{*}{$3200$}  &$A$ & 68026 &real unsymmetric & $2.02\times 10^{44}$\\
 &  && $B$ & 18316& real symmetric indefinite & $2.02\times 10^{13}$\\ \hline
\multirow{2}{*}{6}&\multirow{2}{*}{\texttt{MHD4800}} & \multirow{2}{*}{$4800$} &$A$& 102252& real unsymmetric & $2.54\times 10^{57}$\\
& & & $B$&  27520&real symmetric indefinite &$1.03\times 10^{14}$
\end{tabular}}
\label{Tab:5-1}
\end{table}

We use Gauss-Legendre quadrature rule with $q=16$ quadrature points on $\Gamma$ to compute the contour integrals (\ref{eq:4-13-1}) \cite{DR84}.
As for solving the generalized shifted linear systems of the form (\ref{eq:5-11-1}),
we first use the \textsc{Matlab} function \texttt{lu} to compute the LU
decomposition of $ A-z_j B, j =1 ,2,\ldots, q$, and then perform the triangular
substitutions to get the corresponding solutions. In the experiments, the size of sampling vectors $h_0$ and the parameter $g$ are taken to be $20$ and $5$, respectively.

\textsf{Experiment\ 4.1} The goal of this experiment is to show the convergence behavior of CIQZ. The  test problem is the bounded fineline dielectric waveguide generalized eigenproblem \texttt{BFW782} (cf. \textsc{Table} \ref{Tab:5-1}) \cite{DGL89}. It stems from a finite element discretization of the Maxwell equation for propagating modes and magnetic field profiles of a rectangular waveguide filled with dielectric and PEC structures \cite{BDDRV00}. We are interested in the eigenvalues inside the circle $\Gamma$ with center at $\gamma = -6.0\times 10^5$ and radius  $\rho = 2.0\times 10^5$. By using the \textsc{Matlab} function \texttt{eig} to compute all eigenvalues of the test problem in dense format, we find that there are $141$ eigenvalues within $\Gamma$.

 Define
\begin{equation}\label{eq:7-16-1}
\texttt{max\_r} = \max_{1\leq i \leq s} r_i, 
\end{equation}
where  $r_i$ are the residual norms given by  $ \|A\tilde{{\bf x}}_i-\tilde{\lambda}_i B\tilde{{\bf x}}_i\|/(\|A\tilde{{\bf x}}_i\|+\|B\tilde{{\bf x}}_i\|)$ and $(\tilde{\lambda}_i, \tilde{{\bf x}}_i)$ are the filtered eigenpairs. In our CIQZ method, we stop the iteration process when: (i) the numbers of filtered eigenpairs in two consecutive iterations are the same, and (ii) $\texttt{max\_r}$ in  current iteration is less than the prescribed tolerance $\epsilon$ (step 21 in Algorithm 3).
 
The left picture in \textsc{Fig} \ref{Fig:5-3}  depicts the numbers of  filtered eigenpairs in ten iterations. Recall that the filtered eigenpairs are the ones whose residual norms  are less than the test tolerance $\eta$. In the experiments, we take $\eta = 1.0\times 10^{-3}$. We see that in the first iteration there are 28 approximate eigenpairs whose residual norms are less than $\eta$. But from the second iteration to the last, the number of filtered eigenvalues in each iteration is 141, which is exactly the number of eigenvalues inside $\Gamma$.

\begin{figure}
\begin{center}
\includegraphics[width=13cm]{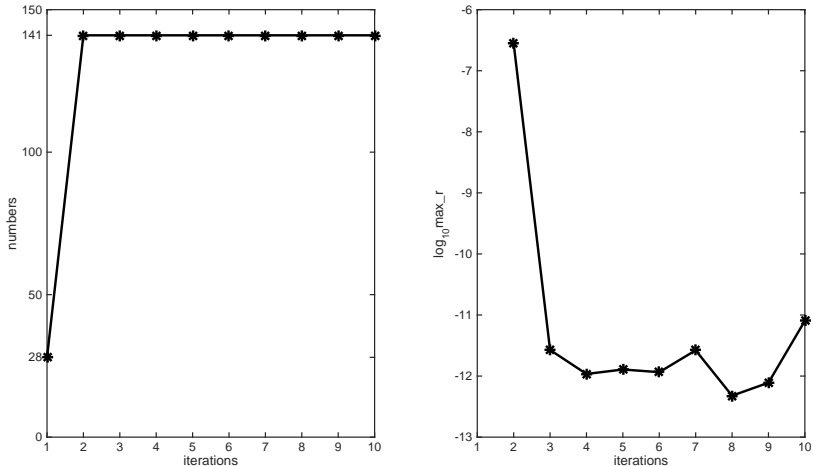}
\caption{The convergence behavior of CIQZ for test problem \texttt{BFW782}.}
\label{Fig:5-3}
\end{center}
\end{figure}

The right picture in  \textsc{Fig} \ref{Fig:5-3} shows the maximum of the residual norms of filtered eigenvalues, i.e., $\texttt{max\_r}$ (cf. (\ref{eq:7-16-1})), in each iteration. From the left picture, we know that the number of filtered eigenvalues attains the one of eigenvalues inside $\Gamma$ starts from the second iteration. Therefore, we plot $\texttt{max\_r}$ starting from the second iteration to the 10th. We see that $\texttt{max\_r}$ decreases monotonically and dramatically from the second iteration to the fourth,  maintains at almost the same level in the next three iterations, and rebounds from the eighth iteration.

\textsf{Experiment\ 4.2}
This experiment is devoted to showing the numerical performance of our CIQZ.  We compare the CIQZ method with \textsc{Matlab} built-in function \texttt{eig} and the  block CIRR method (\textsc{Block\_CIRR}). In  \cite{SFT}, the authors addressed some implementation problems of the block CIRR method, including the selection of the size of the starting vectors and iterative refinement scheme. In the experiment, as for the  block CIRR method,  we use the version proposed in \cite{SFT}.  Note that the dominate computation cost of both CIQZ and \textsc{Block\_CIRR} comes from solving $q = 16$ generalized linear shifted systems of the form (\ref{eq:5-11-1}). On the other hand, when using \texttt{eig} to compute the eigenvalues inside the target region, we have to first compute all eigenvalues in dense format and then select the target eigenvalues according to their coordinates. However, the matrices listed in \textsc{Table} \ref{Tab:5-1} are sparse. Therefore, for the sake of fairness, we compare the three methods only in terms of accuracy, and will not show the amount of CPU time taken by each method.  

\textsc{Block\_CIRR} and CIQZ are contour-integral based eigensolvers, the common parameters $h_0$ and $g$ we take to be $20$ and $ 5$, respectively. In \cite{SFT},  \textsc{Block\_CIRR} performs two iterative refinements, i.e., three iterations in total. For comparison, in the experiment, we set the convergence tolerance $\epsilon = 1.0\times 10^{-15}$ and ${\rm max\_iter} = 3$ for our CIQZ method. As a result, the results computed by CIQZ and \textsc{Block\_CIRR} will actually be those computed in the third iteration.

We use $\texttt{max\_r}$ (cf. (\ref{eq:7-16-1})) to measure the accuracy achieved by each of three methods.  In \textsc{Table} \ref{Tab:5-1}, $\gamma$ and $\rho$ represent the center and the radius of target circle $\Gamma$ respectively, and $s$ is the number of eigenvalues inside $\Gamma$. In the last three columns of \textsc{Table} \ref{Tab:5-1}, we display the $\texttt{max\_r}$s computed by all three methods for each of the six test problems. 

From \textsc{Table} \ref{Tab:5-1}, we see that for the two  contour-integral based eigensolvers, \textsc{Block\_CIRR} and CIQZ, the latter outperforms the former in all six test problems. Especially, as for the problem 7, whose matrices are ill-conditioned, \textsc{Block\_CIRR} fails to compute the desired eigenpairs. Therefore, our CIQZ method is more accurate and reliable than \textsc{Block\_CIRR}. When it comes to the comparison of \textsc{Matlab} function \texttt{eig} and CIQZ,  it is shown that the results  computed by the two methods agree almost the same digits of accuracy for the first three test problems; while our CIQZ method is more accurate than \texttt{eig} by around two digits of accuracy in the last three problems, whose matrices are ill-conditioned. We should point out that in the experiment our CIQZ method just runs three iterations, it may obtain more accurate results if it performs more iterations.

\begin{table}
\centering
\caption{Comparison of   \texttt{eig}, \textsc{Block\_CIRR}, and CIQZ.}
\footnotesize{
\noindent
\begin{tabular}{c|cc|c|ccc}
No.&$\gamma$&$\rho$&$s$& \texttt{eig} &\textsc{Block\_CIRR}&CIQZ  \\ \hline
1& $-5.0\times 10^5$&$1.0\times 10^5$  & 58 &$2.57\times 10^{-13}$&$3.60\times 10^{-9}$&$ 4.76\times 10^{-13}$   \\
2&  $-6.0\times 10^{5}$&$2.0\times 10^{5}$  & 141&$5.59\times 10^{-12}$&$1.45\times 10^{-8}$&$ 2.09\times 10^{-12}$ \\
3&$-5.0\times 10^5$& $3.0\times 10^5$&143&$6.81\times 10^{-10}$& $1.41\times 10^{-6}$ & $5.85\times 10^{-10}$ \\
4&$-1.0\times 10^1$& $8.0$&72&$1.15\times 10^{-8}$&$ 5.23\times 10^{-8}$ &$7.00\times 10^{-10}$ \\
5&$-5.0\times 10^1$& $3.0\times 10^1$&137&$4.94\times 10^{-7}$&$9.77\times 10^{-6}$&$ 1.67\times 10^{-9}$ \\
6&$-5.0$& $3.0$&208&$1.99\times 10^{-6}$&---&$ 1.79\times 10^{-8}$ \\
\end{tabular}}
\label{Tab:5-2}
\end{table}

\section{Conclusions}  
In this paper, we present a new contour-integral based method for computing the eigenpairs inside a given region. Our method is based on the CIRR method. The main difference between the original CIRR method and our CIQZ method is the way to extract the desired eigenpairs. We establish the mathematical framework and address some implementation issues for our new method.  Numerical experiments show that our method is reliable and accurate.


\begin{thebibliography}{10}
\bibitem{Rudin}
{\sc L. Ahlfors}, {\it Complex Analysis}, 3rd Edition,
McGraw-Hill, Inc., 1979.

\bibitem{Lapack}
{\sc A. Anderson, Z. Bai, C. Bischof, S. Blackford, J. Demmel, J. Dongarra, J.~Du Croz, A Greenbaum, S. Hammarling, A. McKenney, and D. Sorensen}, {\it LAPACK Users' Guide}, 3rd Edition, SIAM, Philadephia, 1999.

\bibitem{AT14}
{\sc A.~P. Austin and L.~N. Trefethen}, {\it Computing eigenvalues of real real symmetric matrices with rational filters in real arithmetic}, preprint.

\bibitem{BDDRV00}
{\sc Z. Bai, J. Demmel, J. Dongarra, A. Ruhe, and H. Van Der Vorst}, {\it Templates for the Solution of Algebraic Eigenvalue Problems: A Practical Guide }, SIAM, Philadelphia,  2000.


\bibitem{BGL07}
{\sc B. Beckermann, G. H. Golub, and G. Labahn}, {\it On the numerical condition of a generalized Hankel eigenvalue problem}, Numer. Math., 106 (2007), pp. 41--68.


\bibitem{DGL89}
{\sc R.~F. Boisvert, R. Pozo, K. Remington, R. Barrett, and J. Dongarra}, {\it The matrix market: A web repository for test matrix data}, in The Quality of Numerical Software, Assessment and Enhancement, R. Boisvert, ed., Chapman \&  Hall, London, 1997, pp. 125 --137.

\bibitem{tony}
{\sc T.~T. Chan}, {\it Rank revealing QR factorizations}, Lin. Alg. Appl., 88-89 (1987),  pp. 67--82.


\bibitem{DR84}
{\sc P.~J. Davis and P. Rabinowitz}, {\it Methods of numerical integration}, 2nd Edition, Academic Press, Orlando, FL, 1984.

\bibitem{demmel}
{\sc J. Demmel}, {\it Applied Numerical Linear Algebra}, SIAM, Philadelphia, 1997.

\bibitem{FS12}
{\sc H-R. Fang and Y. Saad}, {\it A filtered Lanczos procedure for extreme and interior eigenvalue problems}, SIAM J. Sci. Comput., 34 (2012), pp. A2220--A2246.


\bibitem{FSV98}
{\sc D.~R. Fokkema, G.~L.~G. Sleijpen, and H.~A. Van Der Vorst}, {\it Jacobi-Davidson style QR and QZ algorithms for the reduction of matrix pencils}, SIAM J. Sci. Comput., 20 (1998), pp. 94--125.

\bibitem{futa}
{\sc Y. Futamura, H. Tadano, and T. Sakurai}, {\it Parallel stochastic estimation method of eigenvalue distribution}, JSIAM Letters 2 (2010), pp.127--130.

\bibitem{GGD96}
{\sc K. Gallivan, E. Grimme, and P. Van Dooren}, {\it A rational Lanczos algorithm for model reduction}, Numer. Algorithms, 12 (1996), pp. 33--64.

\bibitem{G59}
{\sc F. R.~Gantmacher}, {\it The Theory of Matrices}, Chelsea, New York, 1959.

\bibitem{gvl}
{\sc G. H. Golub and C. F. Van Loan}, {\it Matrix Computations}, 3rd Edition, Johns Hopkins University Press, Baltimore, MD, 1996.



\bibitem{IDS15}
{\sc A. Imakura, L. Du, and T. Sakurai}, {\it Error bounds of Rayleigh-Ritz type contour integral-based eigensolver for solving generalized eigenvalue problems}, Numer. Algor., 68 (2015), pp. .

\bibitem{IS10}
{\sc T. Ikegami and T. Sakurai}, {\it Contour integral eigensolver for non-Hermitian systems: a Rayleigh–Ritz-type approach}, Taiwanese J. Math., 14 (2010), pp. 825--837.

\bibitem{ISN10}
{\sc T. Ikegami, T. Sakurai, and U. Nagashima}, {\it A filter diagonalization for generalized eigenvalue problems based on the Sakurai-Sugiura projection method}, J. Comp. Appl. Math., 233 (2010), pp. 1927--1936.




\bibitem{MS73}
{\sc C.~B. Moler and G.~W. Stewart}, {\it An algorithm for generalized matrix eigenvalue problems}, SIAM J. Numer. Anal., 10 (1973), pp. 241--256.


\bibitem{polizzi}
{\sc E. Polizzi}, {\it Density-matrix-based algorithm for solving eigenvalue problems}, Phys. Rev. B 79 (2009) 115112.

\bibitem{Ruhe98}
{\sc A. Ruhe}, {\it Rational Krylov: A practical algorithm for large sparse nonsymmetric matrix pencils}, SIAM J. Sci. Comput., 19 (1998), pp. 1535--1551.

\bibitem{saad}
{\sc Y. Saad}, {\it Numerical Methods for Large Eigenvalue Problems}, SIAM, Philadelphia, 2011.


\bibitem{SFT}
  {\sc T. Sakurai, Y. Futamura, and H. Tadano},
  {\it Efficient parameter estimation and implementation of a contour integral-based eigensolver}, J. Alg. Comput. Tech., 7 (2013), pp. 249--269.

\bibitem{ss}
{\sc T. Sakurai and H. Sugiura}, {\it A projection method for generalized eigenvalue problems using numerical integration}, J. comput. Appl. Math., 159 (2003), pp. 119--128.

\bibitem{ST07}
  {\sc T. Sakurai and H. Tadano},
  {\it CIRR: A Rayleigh--Ritz type method with contour integral for generalized eigenvalue problems},
 Hokkaido Math. J. 36 (2007), pp. 745--757.





\bibitem{stewart}
{\sc G.~W. Stewart}, {\it Matrix Algorithms, Vol. II, Eigensystems}, SIAM, Philadelphia, 2001.




\bibitem{YCY14}
{\sc G. Yin, R. Chan, and M-C. Yueng}, {\it A FEAST algorithm for generalized non-Hermitian eigenvalue problems},  http://arxiv-web3.library.cornell.edu/abs/1404.1768.

\end{thebibliography}
\end{document}